\documentclass{amsart}
\usepackage{amssymb}
\usepackage{graphicx}
\usepackage[shortlabels]{enumitem}
\usepackage{multirow}
\usepackage{amsmath,color}
\usepackage{hyperref}
\usepackage{url}
\usepackage{tikz}
\usepackage{longtable}

\theoremstyle{plain}
\newtheorem{theorem}{Theorem}[section]

\newtheorem{lemma}[theorem]{Lemma}

\theoremstyle{definition}

\theoremstyle{remark}

\def\R{\mbox{$\mathbb{R}$}}

\newcommand{\G}{\Gamma}

\newcommand{\Z}{\mathbb{Z}}

\begin{document}
\title{On a lower bound for the Laplacian eigenvalues of a graph}

\author[G. R. W. Greaves]{ Gary R. W. Greaves }
\thanks{G.R.W.G. was at Tohoku University while this work took place and was supported by JSPS KAKENHI; 
grant number: 26$\cdot$03903}
\author[A. Munemasa]{ Akihiro Munemasa}
\author[A. Peng]{ Anni Peng }


\address{School of Physical and Mathematical Sciences, 
Nanyang Technological University, 21 Nanyang Link, Singapore 637371}
\email{grwgrvs@gmail.com}

\address{Research Center for Pure and Applied Mathematics,
Graduate School of Information Sciences, 
Tohoku University, Sendai 980-8579, Japan}
\email{munemasa@tohoku.ac.jp}

\address{School of Mathematical Sciences, Tongji University, 
1239 Siping Road, Shanghai, P.R. China}
\email{anni\_peng@163.com}



\begin{abstract}
	If $\mu_m$ and $d_m$ denote, respectively, the $m$-th largest Laplacian eigenvalue and the $m$-th largest vertex degree of a graph, then $\mu_m \geqslant d_m-m+2$. 
	This inequality was conjectured by Guo in 2007 and proved by Brouwer and Haemers in 2008. Brouwer and Haemers gave several examples of graphs achieving equality, but a complete characterisation was not given. 
	In this paper we consider the problem
	of characterising graphs satisfying $\mu_m = d_m-m+2$.
	In particular we give a full classification of graphs with $\mu_m = d_m-m+2 \leqslant 1$.
\end{abstract}

\maketitle

\section{Introduction}

Let $\Gamma$ denote a finite, simple graph having $n$ vertices, let $V(\Gamma)$ denote the vertex set of $\Gamma$, and write $x \sim y$ to indicate that the vertices $x$ and $y$ are adjacent.
For a vertex $x$ of $\Gamma$, we write $\deg x$ to denote the \textbf{vertex-degree} of $x$.
The \textbf{adjacency matrix} $A$ of $\Gamma$ is defined as the symmetric $\{0,1\}$-matrix whose rows and columns are indexed by $V(\Gamma)$, where $A_{xy} = 1$ if $x \sim y$ and otherwise $A_{xy} = 0$.
The set of neighbours of a vertex $x \in V(\Gamma)$ is denoted by $\Gamma(x) = \{ v \in V(\Gamma) \mid v \sim x \}$.
The \textbf{Laplacian matrix} $L(\Gamma)$ of $\Gamma$ is defined as $L(\Gamma) = D - A$, where $D$ is the diagonal matrix given by $D_{xx} = \deg x$.
The eigenvalues of $L(\Gamma)$ are known as the \textbf{Laplacian eigenvalues} of $\Gamma$.

Denote by $d_i(\Gamma)$ and $\mu_i(\Gamma)$, respectively, the $i$th largest vertex-degree and $i$th largest Laplacian eigenvalue of $\Gamma$.
When it is clear which graph is under consideration, we merely write $d_i$ and $\mu_i$.

Brouwer and Haemers~\cite[Theorem 1]{BH1} proved the following lower bound for the $m$th largest Laplacian eigenvalue of $\Gamma$. 
\begin{theorem}\label{thm:BH}
	Let $\Gamma$ be a graph having vertex degrees $d_1 \geqslant d_2 \geqslant \dots \geqslant d_n$
	and Laplacian eigenvalues $\mu_1 \geqslant \mu_2 \geqslant \dots \geqslant \mu_n = 0$.
	Suppose $m \in \{1,\dots,n\}$ and $\Gamma \ne K_m \cup (n-m)K_1$.
	Then $\mu_m \geqslant d_m - m +2$.
\end{theorem}

This theorem was conjectured by Guo~\cite{guo}, who had proved the result for the special case when $m=3$.
Special cases of this result had been demonstrated earlier by Li and Pan~\cite{liPan} (who settled the case $m=2$), and Grone and Merris~\cite{ro} (who settled the case $m=1$).

In this paper, we are motivated by the question of, for a given $m \geqslant 1$, which graphs satisfy the equality $\mu_m = d_m - m +2$.
This question was considered in \cite{BH}, however, only partial results were obtained.
In particular, for $m=1$, a connected graph on $n$ vertices satisfies $\mu_1 = d_1+1$ if and only if it has a vertex of degree $n-1$.
Our main results include a full classification of graphs satisfying $\mu_m = d_m - m + 2$ when $\mu_m \leqslant 1$ and a partial classification of graphs satisfying $\mu_m = d_m - m + 2$ for graphs that contain a certain subgraph for $m \geqslant 1$.

In Section~\ref{sec:main_results} we state our main results and we give the proofs in Section~\ref{sec:proofs}.

\section{Main tools and results} 
\label{sec:main_results}

In this section we state our main tools and our main results.
Our main tools are contained in the following two lemmas about the interlacing of eigenvalues.
(See \cite[Section 2]{BH} for proofs.)
 
For a real symmetric matrix $N$ of order $n$, we denote its eigenvalues by
\[\lambda_1(N) \geqslant \lambda_2(N) \geqslant \dots \geqslant \lambda_n(N),\]
and the multiset of the eigenvalues by $\operatorname{Spec}(N)$.

\begin{lemma}[Interlacing I]\label{lem:interlacing1}
Let $N$ be a real symmetric matrix of order $n$.
Suppose that $M$ is a principal submatrix of $N$, 
or a quotient matrix of $N$, of order $m$.
Then the eigenvalues of $M$ \emph{interlace} those of $N$, that is 
$\lambda_i(N) \geqslant \lambda_i(M) \geqslant \lambda_{n-m+i}(N)$ for $i = 1,\dots,m$.
\end{lemma}

\begin{lemma}[Interlacing II]\label{lem:interlacing2}
	Let $\Gamma$ be a graph and let $\Delta$ be a (not necessarily induced) subgraph of $\Gamma$ on $m$ vertices.
	Then $\mu_i(\Delta) \leqslant \mu_i(\Gamma)$ for all $i \in \{1,\dots,m\}$.
\end{lemma}

We will use the phrase ``by interlacing'' to refer to either of the above lemmas.

In our first result we classify the graphs for which $\mu_m=d_m-m+2=0$ for some $m$.

\begin{theorem}\label{thm:0}
Let $\Gamma$ be a graph with $n$ vertices, and let $m \in \{1,\dots,n \}$.
Suppose that $\mu_m=d_m-m+2=0$.
Then one
of the following holds.
\begin{enumerate}
\item $m=2$ and $\Gamma$ is $nK_1$.

\item $m=3$ and $\Gamma$ is $2K_2\cup(n-4)K_1$.

\item $\Gamma$ is $\overline{sK_1\cup tK_2}\cup(n-m)K_1$ for some
    $s,t\in\Z$ with $t>0$ and $s\geqslant 0$.
\end{enumerate}
\end{theorem}

We define an \textbf{$m$-nexus} of a graph $\Gamma$ without isolated vertices
to be an $m$-subset $S$ of $V(\Gamma)$ such that each vertex in $S$ has degree at least $d_m$ and every edge of $\Gamma$ has a nontrivial intersection with $S$.

Let $S$ be an $m$-subset of $V(\Gamma)$ having largest degrees, 
i.e., $\deg v \geqslant d_m(\Gamma)$ for all $v \in S$.
If $\Gamma$ has no isolated vertices and $S$ is not an $m$-nexus,
then there exists an edge $e$ of $\Gamma$ satisfying 
$e \cap S = \emptyset$. 
Let $\Delta$ be the graph obtained from $\Gamma$ by deleting $e$.
Then $S$ is an $m$-subset of $V(\Delta)$ having largest degrees.
If $\Gamma$ satisfies $\mu_m(\Gamma) = d_m(\Gamma)-m+2$,
then $\Gamma\neq K_m\cup K_2\cup(n-m-2)K_1$, and hence
$\Delta\neq K_m\cup(n-m)K_1$. Thus, 
\[
	\mu_m(\Gamma) \geqslant \mu_m(\Delta) \geqslant d_m(\Delta)-m+2 = d_m(\Gamma)-m+2,
\]
where the left and right inequalities follow from Lemma~\ref{lem:interlacing2} and Theorem~\ref{thm:BH} respectively.
This forces $\mu_m(\Delta) = d_m(\Delta) - m + 2$.
Hence we see that any graph $\Gamma$ with $\mu_m(\Gamma) = d_m(\Gamma)-m+2$ 
can be obtained by adding edges to a graph $\Delta$ having an $m$-nexus and $\mu_m(\Delta) = d_m(\Delta) - m + 2$.
The following theorem characterizes such graphs $\Delta$.

\begin{theorem}\label{thm:class1}
	Let $m \in \{1,\dots,n \}$, let $\Gamma$ be a graph 
with $n$ vertices
having an $m$-nexus, and suppose that $\mu_m = d_m - m + 2 \geqslant 1$.
	Then one of the following must hold.
	\begin{enumerate}
		\item $\mu_m = 1$ and $\Gamma$ is $K_m$ with $p$ pendant vertices attached to a vertex of $K_m$;
		\item $\mu_m \geqslant 2$ and $\Gamma$ is $K_m$ with $\mu_m-1$ pendant vertices attached to each vertex of $K_m$;
		\item $m=2$, $\mu_m = d \geqslant 2$, and $\Gamma$ is $K_{2,d}$.
	\end{enumerate}
\end{theorem}

\begin{figure}[htbp]
	\begin{center}
		\begin{tikzpicture}[place/.style={draw,fill=black,shape=circle,inner sep=2pt},
			transition/.style={}, scale=1.5]
		\begin{scope}[scale=1.5,rotate=90,xshift=-0.2cm]

			\foreach \x in {0,...,6}
				\node (v\x) [place] at (\x*51:0.4) {};
			\foreach \x in {0,...,6}
			 {
		 		\foreach \y in {\x,...,6}
					\draw (v\x) -- (v\y);
			 }

			\node (v7) [place] at (-15:0.8) {};
			\node (v8) [place] at (0:0.8) {};
			\node (v9) [place] at (15:0.8) {};
			\foreach \x in {7,8,9}
			    \draw (v0) -- (v\x);
		\end{scope}
		\begin{scope}[scale=1.5]
			\node (t2) [inner sep=0.1pt] at (0,-1) {(1) $m=7$ and $p = 3$};
		\end{scope}
		\begin{scope}[xshift=3.2cm,scale=1.5]
			\foreach \x in {0,...,5}
			    \node (v\x) [place] at (\x*60:0.4) {};
			\foreach \x in {0,...,5}
			 {
		 		\foreach \y in {\x,...,5} 
					\draw (v\x) -- (v\y);
			 }
			\foreach \x in {6,...,17}
			    \node (v\x) [place] at (\x*30+15-180:0.8) {};
			\foreach \x in {0,60,...,300}
			    \draw (\x:0.4) -- (\x+15:0.8);
			\foreach \x in {0,60,...,300}
			    \draw (\x:0.4) -- (\x-15:0.8);
			\node (t2) [inner sep=0.1pt] at (0,-1) {(2) $m=6$ and $\mu_m=3$};
		\end{scope}

		\begin{scope}[xshift=5.7cm,scale=1.5]
			\node (v0) [place] at (.25,.5) {};
			\node (v1) [place] at (.75,.5) {};
			\node (v2) [place] at (0,-0.5) {};
			\node (v3) [place] at (0.33,-0.5) {};
			\node (v4) [place] at (0.67,-0.5) {};
			\node (v5) [place] at (1,-0.5) {};
			\foreach \x in {2,3,4,5}
			    \draw (v0) -- (v\x);
			\foreach \x in {2,3,4,5}
			    \draw (v1) -- (v\x);
			\node (t2) [inner sep=0.1pt] at (0.4,-1) {(3) $d=4$};
		\end{scope}

		\end{tikzpicture}
	\end{center}
	
	\caption{Examples of each of the cases of Theorem~\ref{thm:class1}.}
	\label{fig:examples}
\end{figure}
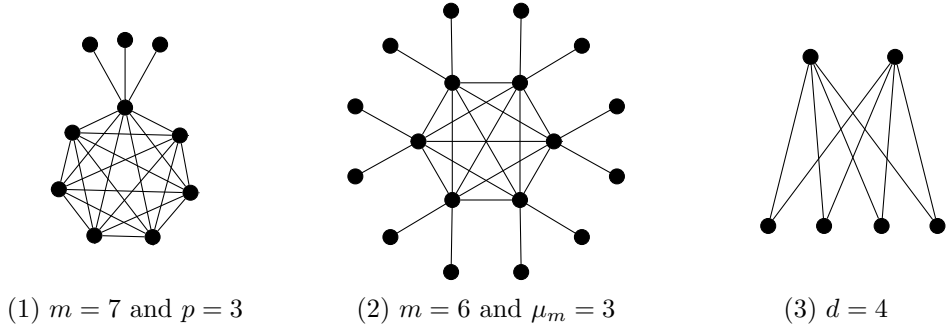

Theorem~\ref{thm:class1} is strengthening of Proposition 3 in \cite{BH1}.
We are interested in graphs for which $\mu_m = d_m-m+2$ for some $m$.
Given such a graph $\Gamma$, by Theorem~\ref{thm:class1} we know that we can delete a sequence of edges to obtain a graph given in Theorem~\ref{thm:class1}.

For the case $\mu_m = d_m-m+2 = 1$ we can give a complete classification:

\begin{theorem}\label{thm:1}
Let $\Gamma$ be a graph with $n$ vertices, and let $m \in \{1,\dots,n \}$.
Suppose that $\mu_m=d_m-m+2=1$.
Then $\Gamma$ is $K_m$ with $p$ pendant vertices attached to a vertex of $K_m$.
\end{theorem}


%
%
%


\section{Proofs of the main results} 
\label{sec:proofs}

In this section we prove our main results.
We begin with the proof of Theorem~\ref{thm:0}.

\begin{proof}[Proof of Theorem~\ref{thm:0}]
Let $c$ be the number of connected components of $\Gamma$.
Since the multiplicity of the Laplacian eigenvalue $0$ is larger than $n-m+1$,
we have $c\geqslant n-m+1$. 
Let $x_m$ be a vertex with degree equal to $d_m$.
Since $d_m=m-2$, the connected component $C$ containing $x_m$ satisfies $|C|\geqslant m-1$.
Hence we have $c\leqslant n-m+2$, which together with the lower bound gives $c \in \{ n-m+1, n-m+2 \}$.

First suppose that $c=n-m+1$. 
Then $|C| \in \{m-1,m\}$.
If $|C|=m-1$ then there are $n-m+1$ vertices outside
$C$ and they constitute $n-m$ components.
It follows that $d_m=1$ and so $m=d_m+2=3$.
This implies that $\Gamma$ is $2K_2\cup (n-4)K_1$.
Otherwise, if $|C|=m$ then there are $n-m$ vertices outside $C$ and they constitute $n-m$
components, so they are all isolated vertices.
Since
$d_1\geqslant \dots \geqslant d_m=m-2$, the graph $\Gamma$ is
$\overline{sK_1\cup tK_2}\cup(n-m)K_1$ for some $s,t\in\Z$ with $s\geqslant 0$ and $t>0$.

Finally suppose that $c=n-m+2$. 
Then $|C|=m-1$ since there are $n-m+1$ vertices
outside $C$. 
So $d_m=0$ and $m=d_m+2=2$. That is, $\Gamma = nK_1$.
\end{proof}

Let $\Gamma$ be a graph and let $S$ be a subset of the vertex set of $\Gamma$.
We write $L_S$ to denote the principal submatrix of $L(\Gamma)$ with rows and
columns indexed by $S$ and we write $L(S)$ to denote the Laplacian of the subgraph of $\Gamma$ induced on $S$.
We will use this notation in some of the lemmas below.

Next we need the following lemma.  
This lemma is a refined version of \cite[Lemma 2]{BH1}.

\begin{lemma}\label{lem:5a}
Let $\Gamma$ be a graph having a subset of $m > 0$ vertices $S$ such that
each vertex in $S$ has at least $e$ neighbours outside $S$.
Then $\mu_m\geqslant e$.
If equality holds, then $S$ is disconnected or $e=0$.
\end{lemma}
\begin{proof}
Firstly we have
\[L_S=L(S)+D,\]
where $D$ is the diagonal matrix with
\[D_{ss}=|\Gamma(s)\setminus S|\geqslant e\quad (s\in S).\]
Since $L(S)$ is positive semidefinite, by \cite[Theorem~2.8.1(iii)]{BH}, we have
\[\lambda_m(L(S)+D)\geqslant\lambda_m(D).\]
Now
\begin{align*}
\mu_m &=\lambda_m(L(\Gamma)) \\
&\geqslant \lambda_m(L_S)
&&\text{(by interlacing)} \\
&= \lambda_m(L(S)+D) \\
&\geqslant
\lambda_m(D) \\
&= \min\{D_{ss}\mid s\in S\} \\
&\geqslant e.
\end{align*}

Note that, equality holds if
and only if
\begin{align}
\lambda_m(L(\Gamma))&=\lambda_m(L_S),\notag\\
\lambda_m(L(S)+D)&=\lambda_m(D),\label{2133i2}\\
\exists s\in S,&\;|\Gamma(s)\setminus S|=e.\label{2133i3}
\end{align}
Suppose that these conditions hold. Then $\lambda_m(D)=e$
by (\ref{2133i3}), so there exists a nonzero vector $\mathbf{u}$ with
$(L(S)+D)\mathbf{u}=e\mathbf{u}$ by (\ref{2133i2}). We may assume without loss
of generality $\mathbf{u}^*\mathbf{u}=1$, and then
\begin{align*}
e&=\mathbf{u}^*(e\mathbf{u}) \\
&= \mathbf{u}^*(L(S)+D)\mathbf{u} \\
&= \mathbf{u}^*L(S)\mathbf{u}+\mathbf{u}^*D\mathbf{u} \\
&\geqslant \mathbf{u}^*D\mathbf{u}
\\&\geqslant
e.
\end{align*}
This forces $\mathbf{u}^*L(S)\mathbf{u}=0$, and hence $L(S)\mathbf{u}=\mathbf{0}$, since
$L(S)$ is positive semidefinite.

Suppose that $S$ is connected and $e>0$.
The latter implies $\mu_m=e>0$, hence $m<n$. 
Thus $S\neq V(\Gamma)$.
Set $r=|V(\Gamma)\setminus S|>0$. Since $S$ is connected, $L(S)\mathbf{u}=\mathbf{0}$ implies
$\mathbf{u}=\frac{1}{\sqrt{m}}\mathbf{1}$. Then $\mathbf{u}^*D\mathbf{u}=e$ implies
$D=eI$. Let $Q$ be the quotient matrix of $L(\Gamma)$
with respect to the partition of $V(\Gamma)$ into
$m+1$ parts:
\[
	\left \{ \{ s \} \mid s \in S \right \} \cup \left \{ V(\Gamma)\setminus S \right \}.
\]
Then
\[Q=\begin{pmatrix}L_S&-e\mathbf{1}\\ -\frac{e}{r}\mathbf{1}^\top&
\frac{em}{r}\end{pmatrix},\]
and, by interlacing,
\begin{equation}\label{2133h}
\mu_m\geqslant\lambda_m(Q)\geqslant\lambda_m(L_S)=e.
\end{equation}

We claim $e\not\in\operatorname{Spec}(Q)$. 
Indeed, suppose
\[Q\begin{pmatrix}\mathbf{v}\\ c\end{pmatrix}=
e\begin{pmatrix}\mathbf{v}\\ c\end{pmatrix},\]
where $\mathbf{v}\in\mathbb{R}^m$, $c\in\mathbb{R}$. 
Then
\begin{align}
L_S\mathbf{v}-ce\mathbf{1}&=e \mathbf{v}, \label{2133h1} \\ 
-\frac{e}{r}\mathbf{1}^\top \mathbf{v}+\frac{cem}{r}&=ce. \label{2133h2}
\end{align}
Since $L_S=L(S)+D=L(S)+eI$, equation~\eqref{2133h1} implies
$L(S)\mathbf{v}=ce\mathbf{1}$. 
Now
\begin{align*}
0&=\mathbf{1}^\top L(S)\mathbf{v}\\
&=\mathbf{1}^\top(ce\mathbf{1}) \\
&= cem.
\end{align*}
This implies $c=0$, and hence $L(S)\mathbf{v}=0$, while
equation~\eqref{2133h2} implies $\mathbf{1}^\top \mathbf{v}=0$.
Since $S$ is connected, we obtain $\mathbf{v}=\mathbf{0}$. 
Therefore, we have
proved the claim. Now equation~\eqref{2133h} implies $\mu_m>e$. This is
a contradiction.
\end{proof}


\begin{lemma}\label{lem:7a5}
Let $c$ and $d$ be positive integers.
Let $\G$ be the graph obtained from $K_{2,d}$ by attaching
$c$ pendant vertices to each of the two non-adjacent vertices 
of degree $d$.
Then $c+d\notin\operatorname{Spec}(L(\G))$.
\end{lemma}

\begin{proof}
If $c=d=1$, then $\G$ is the path of length $4$, so the
proof is straightforward. Assume $c+d\geqslant3$.
Then
\[\det\left((c+d)I-L(\G)\right)=
c(c+d)(c+2d-2)(c+d-1)^{2(c-1)}(c+d-2)^{d-1}
\neq0.\]
\end{proof}

\begin{lemma}\label{lem:7a6}
Let $d$ be an integer with $d\geqslant2$, and
let $\G$ be the graph obtained from $K_{2,d}$ by attaching a
pendant vertex to a vertex of degree $d$. Then
$\mu_2(\G)>d$.
\end{lemma}

\begin{proof}
If $d=2$, then $K_{2,2}$ is the $4$-cycle, so the proof is
straightforward. 
Assume $d>2$. Then
\[\det\left(dI-L(\G)\right)=d^2(d-2)^{d-1}\neq0.\]
Thus $d\notin\operatorname{Spec}(L(\G))$. In particular, $\mu_2(\G)\neq d$.
Since 
$\mu_2(\G)\geqslant\mu_2(K_{2,d})=d$ by interlacing,
we obtain $\mu_2(\G)>d$.
\end{proof}

\begin{lemma}\label{lem:7a}
	Let $\Gamma$ be a graph having $2$-nexus $S=\{x_1,x_2\}$ where $x_1 \not \sim x_2$ and $\deg x_1 \geqslant \deg x_2 \geqslant 1$.
	Suppose that $\mu_2 = d = \deg x_2 $.
Then $\Gamma=K_{2,d}$.
\end{lemma}

\begin{proof}
Let $\Delta$ be the graph obtained from $\G$ be deleting 
$\deg x_1-d$ edges
incident only with $x_1$, so that $x_1$ has degree $d$ in $\Delta$.
Then $\Delta\neq K_2\cup(n-2)K_1$ and $d_2(\Delta)=d=d_2(\G)$.
By interlacing, we obtain $\mu_2(\Delta)=d$.

Let $r=|\Delta(x_1)\cap\Delta(x_2)|$. The graph $\Delta$
is obtained from $K_{2,r}$ by attaching
$d-r$ pendant vertices to each of the two non-adjacent vertices 
of degree $r$.
If $d>r$, then by Lemma~\ref{lem:7a5}, we obtain
$d\notin\operatorname{Spec}(L(\Delta))$, and this contradicts $\mu_2(\Delta)=d$.
Thus $d=r$, and we conclude $\Delta=K_{2,d}$.

Suppose $\G\neq\Delta$. Then $\G$ contains a subgraph $\G'$
obtained from $\Delta=K_{2,d}$ by attaching a
pendant vertex to a vertex of degree $d$. Since
$d_2(\G')=d=d(\G)$, we obtain $\mu_2(\G')=d$ by interlacing.
This contradicts Lemma~\ref{lem:7a6}. Therefore, $\G=\Delta=K_{2,d}$.
\end{proof}

\begin{lemma}\label{lem:6a}
	Let $m \in \{2,\dots,n \}$, let $\Gamma$ be a graph having an $m$-nexus $S$, and suppose that $\mu_m = d_m - m + 2 \geqslant 1$.
	Suppose that each vertex in $S$ has at least $d_m - m + 2$ neighbours in $V(\Gamma) \setminus S$.
	Then $m=2$.
	Furthermore $\Gamma = K_{2,d_m}$ where $d_m \geqslant 2$. 
\end{lemma}

\begin{proof}
Set $d=d_m$ and $e=d-m+1$. 
By Lemma~\ref{lem:5a}, since each vertex in $S$ has $e+1$ neighbours in $V(\Gamma) \setminus S$, the graph induced on $S$ must be disconnected.
Furthermore, there must be at least one vertex $v$ having precisely $e+1$ neighbours in $V(\Gamma) \setminus S$, otherwise, by Lemma~\ref{lem:5a}, we would have $\mu_m \geqslant e+2$, a contradiction.
The vertex $v$ must be adjacent to $m-2$ vertices in $S$.
Hence, since $S$ is disconnected, there exists a unique vertex $s_0$ having no neighbours in $S$.

Now we can write $S$ as the disjoint union $S = \{s_0\} \cup T$. 
Delete edges between $S$ and $V(\Gamma)\setminus S$ so that
every vertex in $T$ has precisely $e$ neighbours outside $S$ and so that $s_0$ has exactly $d$ neighbours outside $S$.
Delete the vertices outside $S$ that are now isolated.
Let $\Delta$ denote the resulting graph.
Let $Q$ be the quotient matrix of $L(\Delta)$
with respect to the partition of $V(\Delta)$ into $(m-1)+1+1=m+1$ parts:
\[
	\left \{ \{ s \} \mid s \in T \right \} \cup \left \{ \{ s_0 \} \right \} \cup \left \{ V(\Delta)\setminus S \right \}.
\]
Then, with $r = |V(\Delta)\setminus S|$, we have
\[Q=
\begin{pmatrix}
L_T & 0 & -(e+1)\mathbf{1} \\
0 & d & -d \\
-\frac{e+1}{r}\mathbf{1}^\top & -\frac{d}{r} & \frac{e(m-1)+d}{r}
\end{pmatrix}.
\]
Whence we obtain a lower bound for $\mu_m(\Delta)$:
\begin{align*}
\mu_m(\Delta)&=\lambda_m(L(\Delta))
\\&\geqslant \lambda_m(Q)\\&\geqslant
\lambda_m{\begin{pmatrix} L_T & 0 \\ 0&d\end{pmatrix}}\\
&= \min\{\lambda_{m-1}(L_T),d\} \\
&=\min\{\lambda_{m-1}(L(T)+(e+1)I),d\}\\
&=\min\{e+1,d\}\\
&=e+1 && \text{ ($d\geqslant e+1$ since $m \geqslant 2$).}
\end{align*}

We claim that if $m>2$ then $e+1\not\in\operatorname{Spec} Q$ and so, by interlacing, $\mu_m(\Gamma)>e+1$. 
Indeed, suppose
\[Q\begin{pmatrix} \mathbf{u} \\ v \\ w \end{pmatrix}=
(e+1)\begin{pmatrix}\mathbf{u} \\ v \\ w \end{pmatrix},\]
where $\mathbf{u}\in\R^{m-1},\;v,w\in\R.$ Then
\begin{align}
L_T\mathbf{u}-(e+1)w\mathbf{1}&=(e+1)\mathbf{u},\label{2133j1}\\
dv-dw&=(e+1)v,\label{2133j2} \\
-\frac{e+1}{r}\mathbf{1}^\top\mathbf{u}-\frac{d}{r}v+\frac{e(m-1)+d}{r}w&=(e+1)w.
\label{2133j3}
\end{align}
Since $L_T=L(T)+(e+1)I$, equation~\eqref{2133j1} implies
$L(T)\mathbf{u}=(e+1)w\mathbf{1}.$ Now
\begin{align*}
0&=\mathbf{1}^\top L(T)\mathbf{u} \\ &=\mathbf{1}^\top(e+1)w\mathbf{1} \\
 &= (e+1)w(m-1).
\end{align*}
This implies $w=0$. Since $T$ is connected, equation~\eqref{2133j1} implies
$\mathbf{u}=k\mathbf{1}$ for some $k\in\R$.
And by equation~\eqref{2133j2}, we have $(d-e-1)v=0$.
If $m>2$ then $d>e+1$, $v=0$ and, by equation~\eqref{2133j3}, we have $\mathbf{u}=\mathbf{0}$.
Hence $e+1\not\in\operatorname{Spec} Q$.

Therefore $m=2$, and by Lemma~\ref{lem:7a}, we have $\Delta = K_{2,d_m}$.
By Lemma~\ref{lem:7a6} and interlacing, we conclude $\Gamma = K_{2,d_m}$.
\end{proof}

\begin{lemma}\label{lem:6b}
	Let $m \in \{2,\dots,n \}$, let $\Gamma$ be a graph having an $m$-nexus $S$, and suppose that $\mu_m = d_m - m + 2 \geqslant 1$.
	Suppose that at least one vertex in $S$ has precisely $d_m - m + 1$ neighbours in $V(\Gamma) \setminus S$.
	If $d_m - m + 1 = 0$ then $\Gamma$ is $K_m$ with $p$ pendant vertices attached to a vertex, otherwise $\Gamma$ is $K_m$ with $\mu_m - 1$ pendant vertices attached to each vertex.
\end{lemma}

\begin{proof}
	Set $e = d_m - m + 1$.
	There exists at least one vertex in $S$ with only
	$e$ neighbours in $V(\Gamma) \setminus S$. 
	Every such vertex is
	adjacent to all other vertices in $S$. 
	Let $W$ be the set of these
	vertices and let $t=|W|$. 
	
	Delete edges between $S\setminus W$ and $V(\Gamma)\setminus S$ so that
	every vertex in $S\setminus W$ has precisely $e+1$ neighbours in $V(\Gamma)\setminus S$.
	Delete the vertices outside $S$ that now are isolated.
	Let $\Delta$ denote the resulting graph.
	
	Consider the quotient matrix $Q$ of $L(\Delta)$ for
	the partition of the vertex set $V(\Delta)$ into $t + (m-t) + 1 = m+1$ parts:
	\[
		\left \{ \{s\} \mid s \in W \right \} \cup \left \{ \{s\} \mid s \in S\setminus W \right \}  \cup \left \{ V(\Delta)\setminus S \right \}.
	\] 
	Let $r=|V(\Delta)\setminus S|$.
	Then
	\[Q=\begin{pmatrix} L_W & -J & -e\mathbf{1} \\ -J^\top & L_{S\setminus W} &
	-(e+1)\mathbf{1} \\-\frac{e}{r}\mathbf{1}^\top &-\frac{e+1}{r}\mathbf{1}^\top &
	\frac{me+m-t}{r}
	\end{pmatrix},
	\]
	where $L_W=(e+m)I-J$.

	Consider the quotient matrix $R$ of $L(\Delta)$ for the partition of the
	vertex set $X$ into $3$ parts $W$, $S\setminus W$, and $V(\Delta)\setminus S$. Then
	\[R=\begin{pmatrix} e+m-t & -m+t &-e \\ -t & e+t+1 & -e-1 \\
	-\frac{te}{r} & -\frac{(e+1)(m-t)}{r} & \frac{me+m-t}{r}\end{pmatrix}.
	\]
	The eigenvalues of $R$ are
	\[0,\frac{A-\sqrt{B}}{2r},\frac{A+\sqrt{B}}{2r},\]
	where
	\begin{align*}
	A&=(e+1)m+(2e+m+1)r-t \\
	B&=(e+1)^2m^2 + (m-1)^2r^2 - 2(e + 1)m(m-1)r \\
	 &\quad  - 2((e + 1)m + (2e - m + 1)r - 2r^2)t + t^2.
	\end{align*}
	These three eigenvalues are also the eigenvalues of $Q$ whose eigenvectors
	are constant on three sets $W$, $S\setminus W$, and $V(\Delta)\setminus S$.

	The left eigenvectors corresponding to the remaining eigenvalues
	of $Q$ are perpendicular to the subspaces that are constant
	on the three sets.
	Hence the eigenvectors are of the form
	$(\mathbf{u}^\top,\mathbf{v}^\top,0)$ with
	$\mathbf{1}^\top\mathbf{u}=\mathbf{1}^\top\mathbf{v}=0$.
	Therefore these remaining eigenvalues of $Q$ are eigenvalues of
	\[P=\begin{pmatrix}L_W & -J\\-J^\top&L_{S\setminus W}\end{pmatrix},\]
	which is a principal submatrix of $Q$.
	Furthermore, these eigenvalues remain
	unchanged if a multiple of $J$ is added to a block of the partition
	of $P$. 
	So they are also the eigenvalues of the matrix
	\[P'=\begin{pmatrix} (e+m)I & O  \\ O & L_{S\setminus W}\end{pmatrix}.
	\]
	Since $L_{S\setminus W}=L(S\setminus W)+(e+t+1)I$ and
	$L(S\setminus W)$ is positive semidefinite, the eigenvalues $\theta$
	of $P'$
	satisfy
	$\theta\geqslant e+t+1>e+1$.
	
	It remains to consider the eigenvalues $\frac{A-\sqrt{B}}{2r}$ and $\frac{A+\sqrt{B}}{2r}$.
	Observe that we have
\begin{align}
	(A-2r(e+1))^2-B 
&	=4r(m-t)\left(e(m-1)+(m-(t+1))\right.
\notag\\&\quad\left.+t(me+m-t-r)\right).
\label{2133t1}
\end{align}
	Now assume that, for $e \ne 0$, we have $t<m$ and, for $e = 0$, we have $t < m-1$.
	Then, if $e\neq 0$, since $t<m$ we see that equation~\eqref{2133t1} is positive.
	And if $e=0$, since $m>t+1$ again we see that equation~\eqref{2133t1} is  positive.
	Hence, in either case, we find that $(A-2r(e+1))^2-B > 0$, which implies $(A-\sqrt{B})/2r > e+1$.
	Therefore, except for the smallest one, all eigenvalues of $Q$ are strictly larger than $e+1$.
	By interlacing, we would have $\mu_m>e+1$, a contradiction.
	
	Therefore, if $e \ne 0$ then we must have $t = m$.
	Further $r = em + m - t = em$, that is, $\Delta$ is a complete graph with $e = \mu_m - 1$ pendant vertices attached to each vertex.
	In this case $\Gamma = \Delta$.
	
	And finally, if $e = 0$ then we must have either $t = m-1$ or $t = m$.
	For $m=t$ the graph $\Delta$ is $K_m$ and so is $\Gamma$, but for such a graph $\mu_m \ne d_m - m + 2$.
	For $m=t+1$, we have $r\geqslant 1$, that is, the graph $\Delta$ is a complete graph $K_m$ with $r$ pendant vertices attached at the same vertex. 
	Hence $\Gamma$ must be a complete graph with $p$ pendant vertices attached at the same vertex, where $p \geqslant r$. 
	This completes the proof.
\end{proof}

Now we can prove Theorem~\ref{thm:class1}.

\begin{proof}[Proof of Theorem~\ref{thm:class1}]
	Since graphs with $\mu_1 = d_1 + 1$ have already been classified, we assume that $m\geqslant 2$.
Let $e=d_m-m+1\geqslant 0$ and let $S$ be an $m$-nexus of $\Gamma$. 
Each vertex in $S$ has at least $e$ neighbours outside $S$. 
First suppose that each vertex of $S$ has at least $e+1$ neighbours outside.
Then, by Lemma~\ref{lem:6a}, we have $\Gamma=K_{2,d_m}$ where $d_m \geqslant 2$.
Finally suppose at least one vertex of $S$ has precisely $e$ neighbours outside.
Then the rest of the theorem follows from Lemma~\ref{lem:6b}.
\end{proof}

Finally we use Theorem~\ref{thm:class1} to prove Theorem~\ref{thm:1}.

\begin{proof}[Proof of Theorem~\ref{thm:1}]
%
Let $S$ be an $m$-subset of the vertices of $\Gamma$ such that each vertex in $S$ has degree at least $d_m$.

Suppose first that $S$ is an $m$-nexus. 
Let $G(r)$ be a complete graph $K_m$ with $r$ pending
edge attached at the same vertex, where $r\geqslant 1$.
By Theorem~\ref{thm:class1}, the graph $\Gamma$ is exactly $G(r)$ for some $r$.

Now assume that there are some edges entirely outside $S$. If
$\Gamma$ has another component, then since, for any $r \geqslant 1$, we have $\mu_m(G(r)\cup K_2)=2>1$,
by interlacing, $\mu_m(\Gamma)>1$. So $\Gamma$ has only one connected
component. Hence there must exist a vertex
$u\in\bigcup_{v\in S}\Gamma(v)\setminus S$ such that $\deg u \geqslant 2$.
Delete the edges attached to $u$ such that $u$ is still the neighbour
of $S$ and $\deg u =2$.

Consider the $(m+1)\times(m+1)$ principal submatrix $Q$ indexed
by $S$ and $u$. Without loss of generality, we assume that $u$ is
adjacent to the first index of $S$, denoted it by $v$. Then
\[Q=\begin{pmatrix}e+m-1&-1&\cdots&-1&-1\\-1&&&&0\\
\vdots&&mI-J&&\vdots\\-1&&&&0\\-1&0&\cdots&0&2\end{pmatrix}.\]
Consider the quotient matrix $R$ of $Q$ for the partition of the
vertex set into three parts $v, S\setminus v$ and $u$. We find
\[R=\begin{pmatrix}e+m-1&1-m&-1\\-1&1&0\\-1&0&2\end{pmatrix}.\]

The eigenvalues of $R$ are the roots of
\[f(x)=x^3-(e+m+2)x^2+(3e+2m-1)x-2e+1.\]
These three numbers are also the eigenvalues of $Q$ for eigenvectors
that are constant on three sets $\{v\},S\setminus \{v\},\{u\}.$

The eigenvectors corresponding to the remaining eigenvalues of $Q$
are
perpendicular to the subspace that are constant on the three sets,
so of the form $(0,\mathbf{w},0)$, with $\mathbf{1}^\top \mathbf{w}=0,$ so they are the
eigenvalues of the following principal submatrix of $Q$: $mI-J$.
And these eigenvalues remain unchanged if a multiple of $J$ is
added. So they are also the eigenvalues of $mI$. Since $m>1$, the
remaining eigenvalues are strictly larger than $1$.

Since $f(1)=m-1>0$ and $f(3)=-2r-3m+7<0$, the second largest root
of $f(x)$ are strictly larger than $1$. So, by interlacing,
$\mu_m(\Gamma)>1$.

This implies that there are no edges entirely outside $S$, so
$\Gamma$ is exactly $G$.
\end{proof}

\end{document}